\documentclass[11pt]{amsart}

\newtheorem{thm}{Theorem}[section]
\newtheorem*{thmn}{Theorem}
\newtheorem{lem}[thm]{Lemma}

\theoremstyle{definition}
\newtheorem{defn}[thm]{Definition}
\theoremstyle{remark}
\newtheorem{ques}[thm]{Question} 
\newtheorem*{rem}{Remark}
\numberwithin{equation}{section}

\newcommand{\caseface}[1]{\textbf{#1}}

\newcommand{\seq}[2]{( #1_{#2})_{#2=1}^\infty} 
\newcommand{\indexed}[3]{( #1_{#2})_{#2\in#3}} 
\newcommand{\finseq}[3]{( #1_{#2})_{#2=1}^{#3}} 
\newcommand{\subseq}[3]{( #1_{ {#2}_{#3} })_{#3=1}^\infty} 
\newcommand{\finsubseq}[4]{( #1_{ {#2}_{#3} })_{#3=1}^{#4}}
\newcommand{\rsphere}{\mathbb{C}_\infty} 
\newcommand{\nat}{\mathbb{N}}

\newcommand{\Disk}{\mathbb{D}} 
\newcommand{\sphere}{\mathbb{S}^2} 
\newcommand{\ucirc}{\bd \Disk}
\newcommand{\bd}{\partial} 
 
\newcommand{\Int}{\mathrm{Int}}

\begin{document}

\title[Indecomposable Continua in Surfaces] {Recognizing Indecomposable Subcontinua of Surfaces from Their Complements} 

%    Information for first author:
\author[C.~P.~Curry]{Clinton P.~Curry} 
\address[Clinton P.~Curry] {Department of Mathematics\\
University of Alabama at Birmingham\\
Birmingham, AL 35294-1170}
%    Current address (if needed): 
%\curraddr{}
\email[Clinton P.~Curry]{clintonc@uab.edu}
\thanks{The author was supported in part by NSF-DMS-0353825.}

%    Information for second author (if needed): 
%\author{Author Two}
%\address{}
%\email{}

%    General info
 %Topology of 2-manifolds


\begin{abstract}
	We prove two theorems which allow one to recognize indecomposable subcontinua of closed surfaces without boundary.
	If $X$ is a subcontinuum of a closed surface $S$, we call the components of $S \setminus X$ the \emph{complementary domains of $X$}.
	We prove that a continuum $X$ is either indecomposable or the union of two indecomposable continua whenever it has a sequence $\seq U n$ of distinct complementary domains such that $\lim_{n \rightarrow \infty} \bd U_n = X$.
	We define a slightly stronger condition on the complementary domains of $X$, called the \emph{double-pass condition}, which we conjecture is equivalent to indecomposability.
	We prove that this is so for continua which are not the boundary of one of their complementary domains.
\end{abstract}

\maketitle

\section{Introduction}\label{sec:introduction}

For us, a \emph{continuum} is a compact connected metric space. 
A \emph{closed surface} is a compact and connected, but not necessarily orientable, 2-manifold without boundary. 
We are interested in conditions which imply that a subcontinuum of a closed surface is topologically complicated. 
Further, we would like for these conditions to rely upon how the continuum is embedded in its ambient space, rather than to use internal characteristics of the continuum.

We use \emph{indecomposability} as a criterion for topological complexity. 
A continuum $X$ is \emph{decomposable} if it can be written as the union of a pair of proper subcontinua $A$ and $B$. 
The pair $(A, B)$ is then called a \emph{decomposition} of $X$. 
A continuum which is not decomposable is called \emph{indecomposable}.

Non-degenerate indecomposable continua (i.e., indecomposable continua consisting of more than a point) are certainly quite complicated and have a rich internal structure. 
For example, a non-degenerate continuum $X$ is indecomposable if and only if, for all $p \in X$, the \emph{composant} of $p$
	\[C_p = \bigcup \{\text{proper subcontinua of $X$ containing $p$}\}\]
is a dense set of first category \cite{HocYou61}, in which case there are uncountably many disjoint composants. 
In contrast, decomposable continua have either one or three composants. 

The rich structure of indecomposable continua provides many interesting internal methods of recognizing indecomposability.
%\footnote{In fact, one may assert that the structure is internal in a very concrete sense.  
%For instance, define a composant $C$ of a planar continuum to be \emph{external} if there is a continuum $A \subset \sphere$ such that $A \cap C \neq \emptyset$ and $A$ does not hit every other composant of $X$.  
%The union of all external composants of a non-degenerate indecomposable continuum is a set of first category \cite{Kra74}, indicating that the view from the complement is somehow very thin in the whole continuum.  
%This is one reason that recognition of indecomposable continua from the complement initially appears less fruitful.}
We are instead interested in recognizing indecomposability based on the continuum's interaction with the space in which it lies. 
This approach is in the spirit of some classical work of Kuratowski \cite{Kur68}, Rutt \cite{Rut35}, and Burgess \cite{Bur51}.

In Section~\ref{sec:priorWork}, we recall some earlier theorems which the current work extends. 
In Section~\ref{sec:burgess}, we extend a theorem of C.~E. Burgess about planar continua.
In Section~\ref{sec:charThm} we extend the characterization of planar indecomposable continua in \cite{CurMayTym07} to certain subcontinua of closed surfaces -- those subcontinua which are not the boundary of any of their complementary domains. 
Finally, in Section~\ref{sec:conclusion} we state questions and conjectures.

I would like to thank my advisor, Dr.~John Mayer, for many helpful discussions
and tireless proofreading of preliminary drafts.  I would also
like to thank the referee for careful attention that improved the quality of
this paper.

\section{Prior Work and Notions}\label{sec:priorWork} 

Let $S$ be a closed surface, and $X \subset S$ be a continuum. 
A component of $S \setminus X$ is called a \emph{complementary domain} of $X$.  
If $X$ equals the boundary of one of its complementary domains, then $X$ is called \emph{unshielded}; otherwise, it is \emph{shielded}.  

In order to shorten some statements, we say that a continuum $X$ is \emph{2-indecomposable} if $X$ is decomposable but cannot be written as the essential union of three proper subcontinua. 
By a theorem of Burgess~\cite[Theorem 1]{Bur51}, a 2-indecomposable continuum is in some sense uniquely decomposable. 
Specifically, there exists a decomposition $(A, B)$ of $X$ where $A$ and $B$ are indecomposable continua. 
Further, if $(C, D)$ is any other decomposition, then $C$ and $D$ each contain exactly one of $A$ and $B$.

The first and most famous result in the vein of this work is by Kuratowski. 
It deals exclusively with continua which are unshielded.
\begin{thm} [\cite{Kur68}] \label{thm:kuratowski} 
	Let $X$ be a planar continuum. 
	If $X$ is the common boundary of three of its complementary domains, then $X$ is either indecomposable or 2-indecomposable. 
\end{thm}

C.~E. Burgess generalized this and other related theorems. 
The following is a corollary to the main result in \cite{Bur51}.  
It is an improvement over Theorem~\ref{thm:kuratowski} in that it can detect indecomposability in shielded continua.
\begin{thm} [\cite{Bur51}] \label{thm:burgess}
	Let $X$ be a planar continuum, and let $\seq U n$ be a sequence of distinct complementary domains of $X$. 
	If $X = \lim_{n \rightarrow \infty} \bd U_n$, then $X$ is either indecomposable or 2-indecomposable. 
\end{thm}
Our Theorem~\ref{thm:burgessSurface} extends this theorem.  
It states that the above holds not only for continua in the plane, but for continua in all closed surfaces.

To state the next theorem, we need some terminology from \cite{CurMayTym07}. 
We slightly extend the definitions to apply to an arbitrary closed surface, whereas the work from \cite{CurMayTym07} applies only to $\sphere$. 
\begin{defn}[generalized crosscut]
	Let $S$ be a closed surface and $U$ a connected open subset with non-degenerate boundary. 
	A \emph{generalized crosscut} of $U$ is a homeomorphic copy $A \subset U$ of the open interval $(0,1)$ such that $\overline A \setminus A \subset \bd U$. 
\end{defn}
Generalized crosscuts are similar to crosscuts from prime end theory, except that the closure of a crosscut is by definition a compact arc.  
The closure of a generalized crosscut, however, need not be locally connected.
\begin{defn}[shadow]
	Let $A$ be a generalized crosscut of $U$.  
	A component $V$ of $U \setminus A$ is called a \emph{crosscut neighborhood} of $A$.  
	The \emph{shadow} of $A$ corresponding to $V$ is the set $\overline V \cap \bd U$.
\end{defn}
Observe that a generalized crosscut has either one or two shadows.
We can now state the \emph{double-pass condition}.
\begin{defn}[double-pass condition]
	Let $X \subset S$ be a continuum, and let $\seq U n$ be a sequence of complementary domains of $X$.  We say that $\seq U n$ satisfies the \emph{double-pass condition} when, for any choice of generalized crosscuts $K_n$ of $U_n$, there exists a sequence of shadows $S_n$ of $K_n$ so that $\lim_{n \rightarrow \infty} S_n = X$.
\end{defn}

The main result of \cite{CurMayTym07} is the following theorem.
\begin{thm}[\cite{CurMayTym07}]\label{thm:planchar}
	A continuum $X \subset \sphere$ is indecomposable if and only if it has a sequence $\seq U n$ of complementary domains satisfying the double-pass condition.
\end{thm}
We extend this theorem in Section~\ref{sec:charThm}.  
Specifically, we prove that a \emph{shielded} continuum in a surface is indecomposable if and only if it has a sequence of complementary domains satisfying the double-pass condition.

\section{An Extension of a Theorem of Burgess}\label{sec:burgess}

In this section, we extend Theorem~\ref{thm:burgess} for planar continua to continua in arbitrary closed surfaces.  
Namely, we prove that a continuum in a surface which is the limit of boundaries of \emph{distinct} complementary domains is either indecomposable or 2-indecomposable.

\subsection{Graphs and Cat's Cradles}

The following lemma gives us a convenient way to gain information about the genus of a particular surface. 

%begin sug1
\begin{defn}
  The \emph{complete bipartite graph} $K_{m,n}$ is the graph with $m+n$ vertices, $\{a_1, \ldots, a_m, b_1, \ldots, b_n\}$, and $mn$ edges joining every vertex in $\{a_1, \ldots, a_m\}$ to every vertex in $\{b_1, \ldots, a_n\}$.
\end{defn}
%end sug1
	 
\begin{lem} \label{lem:bipartite} 
	If $S$ is a closed surface containing an embedding of the graph $K_{3, 4g-1}$, then the genus of $S$ is at least $g$. 
\end{lem}

\begin{proof}
	According to \cite{Frank1969}, the minimal genus $g$ of any surface in which the graph $K_{m, n}$ can be embedded is $\left \lceil \frac{(m-2)(n-2)}{4} \right \rceil$.  (Here, $\lceil x \rceil$ denotes the smallest integer $n$ such that $x \le n$.)  Substituting $m = 3$ and $n = 4g-1$ gives that the minimal genus for a surface to contain a copy of $K_{3,4g-1}$ is equal to $\left \lceil g - \frac 3 4 \right \rceil = g$. 
\end{proof}

For the remainder of this section, let $S$ denote a particular closed surface of genus $g$.  
Below we will use the following convenient notation:  For an arc $A$, let $\Int(A)$ denote the open sub-arc of $A$ joining the endpoints of $A$.
\begin{defn}[cat's cradle]
	Let $d_1, d_3 \in S$ be distinct points. A collection of arcs $\{A_\alpha\}_{\alpha \in J}$ is called a \emph{cat's cradle (between $d_1$ and $d_3$)} if 
	\begin{enumerate}
		\item $d_1$ and $d_3$ are the endpoints of each $A_\alpha$ and 
		\item for distinct $\alpha, \beta \in J$, $\Int(A_\alpha) \cap \Int(A_\beta) = \emptyset$. 
	\end{enumerate}
	In addition, if each $A_\alpha$ intersects a set $D_2$ disjoint from $\{d_1, d_3\}$, then $\{A_\alpha\}_{\alpha \in J}$ is a cat's cradle \emph{through $D_2$}. 
\end{defn}

We now prove a general fact about a cat's cradle through a closed disk in a surface.  
On the face of it, there are many ways in which the arcs of a cat's cradle could be arranged.
It is certainly possible for the arcs to intersect the disk $D_2$ in parallel chords (in a way that one might call \emph{linearly ordered}, defined precisely below).  
On the other hand, the arcs could be arranged in such a way that no chord in the intersection divides the disk between any other two chords in the intersection.  
The following lemma shows that this second possibility does not occur often in a surface of finite genus.

\begin{lem} \label{lem:meetingComponent}
	Let $d_1, d_3 \in S$, and let $D_2 \subset S$ be a closed disk.  Suppose $\finseq{A}{n}{4g+3}$ is a cat's cradle between $d_1$ and $d_3$ through $D_2$, where $g$ is the genus of $S$. Then no component of $D_2 \setminus \bigcup_{n=1}^{4g+3} A_n$ can have closure which meets every element of $\finseq{A}{n}{4g+3}$. 
\end{lem}
\begin{proof}
  Suppose some component $V$ of $D_2 \setminus \bigcup_{n=1}^{4g+3} A_n$ has closure which meets every arc $A_n$. Designate a point $d_2 \in V$, and choose arcs $\finseq{A'}{n}{4g+3}$ satisfying the following for all different indices $m$ and $n$:
  \begin{enumerate}
    \item $A'_n$ joins $d_2$ to a point of $A_n$, with no proper sub-arc of $A'_n$ doing so;
    \item $A'_n \cap A'_m = \{d_2\}$; and
    \item $\Int(A'_n) \subset V$ for each $n$.
  \end{enumerate} 
  The graph $\bigcup_{n=1}^{4g+3} (A_n \cup A_n')$ is then homeomorphic to the graph $K_{3, 4g+3}$, being the complete bipartite graph between the sets $\{d_1, d_2, d_3\}$ and $\bigcup_{n=1}^{4g+3} (A_n \cap A_n')$.  Lemma~\ref{lem:bipartite} indicates that the genus of $S$ is at least $g+1$, contradicting the assumption that it is $g$. 
\end{proof}

Lemma~\ref{lem:linearOrder} will be used in the proof of Theorem~\ref{thm:burgessSurface} to build interesting graphs in $S$.  To state it, we need a definition.
\begin{defn}[linearly ordered]
	Let $ \indexed{A'}{\alpha}{J}$ be an ordered sequence of pairwise disjoint compact arcs in a closed disk $D_2$, irreducible with respect to intersecting $\bd D_2$ twice. 
	We say that $ \indexed{A'}{\alpha}{J}$ is \emph{linearly ordered} if, for $\alpha < \beta < \gamma$ in $J$, $A'_\alpha$ is separated from $A'_\gamma$ in $D_2$ by $A'_\beta$. 
	If $J$ is finite, we call a component of $D_2 \setminus \bigcup_{n \in J} A'_n$ an \emph{end component} if its closure meets only one element of $ \indexed {A'} n J$. 
\end{defn}

In what follows a \emph{chord} of a disk $D_2$ is an arc in $D_2$ which intersects $\bd D_2$ exactly at its endpoints.

\begin{lem} \label{lem:linearOrder} 
	Let $\seq A n$ be a cat's cradle from $d_1$ to $d_3$ through a closed disk $D_2$, and suppose that $A_n \cap \bd D_2$ is finite for each $n$. 
	Then there is a subsequence $\subseq{A}{n}{i}$ and a collection of sub-arcs $\subseq{A'}{n}{i}$, $A_{n_i}' \subset A_{n_i} \cap D_2$, so that $\subseq{A'}{n}{i}$ is linearly ordered in $D_2$. 
\end{lem}
\begin{proof}
	First, using Lemma~\ref{lem:meetingComponent}, notice that only finitely many components of $D_2 \cap \bigcup_{n=1}^\infty A_n$ do not meet $D_2$ in its interior.  
  Hence, by passing to a subsequence, we may assume that $D_2 \setminus A_i$ is not connected for any $i \in \nat$.  
  In particular, $A_i \cap D_2$ contains a non-degenerate chord for every $i \in \nat$ and therefore separates $D_2$.

	There are two cases.  
	
	\caseface{Case 1.}  One case is that, for any element $A_{\alpha}$, there is an element $A_{\beta}$ such that no element $A_{\gamma}$ separates $A_{\alpha}$ from $A_{\beta}$ in $D_2$.
	We will build by induction an increasing sequence $(n_k)_{k=1}^\infty$ and a corresponding sequence of chords $(A'_{n_k})_{k=1}^\infty$ with $A'_{n_k} \subset A_{n_k} \cap D_2$ such that, for all $k \ge 1$, the following two conditions are met:
	\begin{enumerate}
		\item \label{item:linearlyOrdered} The chords $A'_{n_1}, \ldots, A'_{n_k}$ are linearly ordered; and
		\item \label{item:endComponent} some end component of $D_2 \setminus \bigcup_{i=1}^k A'_{n_i}$ bounded by $A'_{n_k}$ intersects infinitely many elements of $(A_n)_{n > n_k}$.  (Notice that there is a unique such end component if $k \neq 1$.)
	\end{enumerate}
	
	For the base case of our induction, set $n_1 = 1$, and let $A'_{n_1} \subset A_{n_1}$ be any non-degenerate chord in $D_2$.
	There are then two components to $D_2 \setminus A'_{n_1}$, each an end component of $D_2 \setminus A'_{n_1}$. 
	Infinitely many elements of the collection $(A_n)_{n \ge n_1}$ must intersect one of them.  
	We have therefore satisfied the requirements for $k=1$.
	
	Now suppose that chords $A'_{n_1}$, \ldots, $A'_{n_k}$ have been found which satisfy conditions \ref{item:linearlyOrdered} and \ref{item:endComponent}.
  Consider the collection $\{A_{N_1}, \ldots, A_{N_m}\}$ of arcs which are not separated from $A'_{n_k}$ by any other $A_n$.
	(There may only be finitely many by Lemma~\ref{lem:meetingComponent}.)
	Let $V$ denote the end component of $D_2 \setminus \bigcup_{i=1}^k A'_{n_i}$ meeting $A'_{n_k}$
	
	Because each $A_{N_i}$ intersects $\bd D_2$ in finitely many points, we see that $\bigcup_{i=1}^m A_{N_i}$ divides $V$ into finitely many components. 
  By choice of $N_1, \ldots, N_m$, no member of $(A_n)_{n > \max(N_1, \ldots, N_m)}$ intersects the component of $V \setminus \bigcup_{i=1}^m A_{N_i}$ which meets $A'_{n_k}$.
	Therefore, infinitely many members must intersect another component $W$ of $V \setminus \bigcup_{i=1}^m A_{N_i}$.
	Let $A'_{n_{k+1}}$ be a minimal sub-arc of $\bigcup_{i=1}^m A_{N_i}$ which separates $W$ from $A'_k$.  
	Then $W$ is contained in the end component of $D_2 \setminus \bigcup_{i=1}^{k+1} A_{n_i}'$ meeting $A_{n_{k+1}}'$.
	That end component meets infinitely many members of $(A_n)_{n > n_{k+1}}$, so we have extended our linearly ordered collection to $k+1$ elements.
	
	By induction, we have a sequence of chords $(A'_{n_i})_{i=1}^\infty$.  
	Since each finite subsequence $(A_{n_i})_{i=1}^k$ is linearly ordered, $(A_{n_i})_{i=1}^\infty$ is linearly ordered.
	
	\caseface{Case 2.}  There exists $A_\alpha$ such that any $A_\beta$ is separated from $A_\alpha$ in $D_2$ by some $A_\gamma$.
	
	Choose $n_1 > \alpha$, and let $A'_{n_1} \subset A_{n_1}$ be any non-degenerate chord.
	Find a sequence $(A_{n_i})_{i=1}^\infty$ such that $A_{n_i}$ is separated from $A'_{\alpha}$ by $A_{n_{i+1}}$.
  Recall that $D_2$, as a two-dimensional disk, is \emph{unicoherent}, which implies that any closed set which separates two points of $D_2$ also has a component which does so.
  Therefore, a chord $A'_{n_{i+1}} \subset A_{n_{i+1}}$ also separates $A_{n_i}$ from $A'_{\alpha}$.
	This process inductively yields the ordered sequence $(A_{n_i})_{i=2}^\infty$, which is evidently linearly ordered.
\end{proof}

\subsection{Burgess's Theorem for Surfaces}

Now we move to the specific setting of a continuum with infinitely many complementary domains. 
\begin{thm}[Burgess's Theorem for Surfaces] \label{thm:burgessSurface} 
	Let $X$ be a subcontinuum of a closed surface $S$, and let $\seq U n$ be a sequence of distinct complementary domains of $X$.
	If $X = \lim_{n \rightarrow \infty} \bd U_n$, then $X$ is either indecomposable or 2-indecomposable.
\end{thm}
\begin{proof}
	By \cite[Theorem 1]{Bur51}, we must only show that $X$ is not the essential union of three proper subcontinua. By way of contradiction, suppose that $X = X_1 \cup X_2 \cup X_3$, where each $X_i$ is a proper subcontinuum of $X$ not contained in $\bigcup_{j \neq i} X_j$.  
	We will first find appropriate closed disks $D_1$, $D_2$, and $D_3$.  
	Designated points in $D_1$ and $D_3$ will function as the endpoints of a cat's cradle through $D_2$. 

	By the definitions of $X_1$, $X_2$, and $X_3$, there exist closed disks $D_1, D_2, D_3 \subset S$ such that 
	\begin{enumerate}
		\item the interior of $D_i$ intersects $X_i$, and
		\item $D_i$ is disjoint from $D_j \cup X_j$ when $i \neq j$. 
	\end{enumerate}
	Because of the increasing density of $\seq{U}{n}$, $U_n$ intersects each $D_i$ when $n$ is large enough. 
	By passing to a subsequence, assume that $U_n \cap D_i \neq \emptyset$ for each $n \in \nat$ and $i \in \{1,2,3\}$. 
	Then for each $n \in \nat$ let $A_n \subset U_n$ be an arc such that, for each $i \in \{1,2,3\}$, 
	\begin{enumerate}
		\item $A_n \cap D_i \neq \emptyset$, 
		\item no proper sub-arc of $A_n$ intersects each $D_i$, and 
		\item $A_n \cap \bd D_i$ is finite. 
	\end{enumerate}
	Note that the last condition can be achieved since each $U_n$ is open.
	
	It is evident that $\Int(A_n)$ intersects only one $D_i$ for each $n$. 
	By passing to a subsequence and relabeling the disks, we can assume without loss of generality that $\Int(A_n) \cap D_2 \neq \emptyset$. 
	Extend each $A_n$ to an arc $\tilde A_n$ with arcs in $D_1$ and $D_3$ to obtain a cat's cradle from $d_1$ to $d_3$ through $D_2$, where $d_i$ is a designated point of $D_i$.  
	Notice that $\bigcup_{n=1}^\infty \tilde A_n$ is still disjoint from $X_2$, since $\tilde A_n \setminus (D_1 \cup D_3) \subset U_n$ is disjoint from $X$ and $\tilde A_n \cap (D_1 \cup D_3) \subset D_1 \cup D_3$ is disjoint from $X_2$.
	
	Let $\subseq{A'}{n}{i}$ be the linearly ordered sequence of compact arcs in $\left(\bigcup_{n=1}^\infty \tilde A_n \right) \cap D_2$ guaranteed by Lemma~\ref{lem:linearOrder}. 
	First, notice that $X_2$ separates $A_{n_i} \cap D_2$ from $A_{n_j} \cap D_2$ in $D_2$ when $i \neq j$.
	To see this, observe that $A_{n_i} \cap D_2 \subset U_{n_i}$, and analogously $A_{n_j} \subset U_{n_j}$.
	These are distinct complementary domains of $X$, so $X$ separates $A_{n_i}$ from $A_{n_j}$ in $S$ and thus in $D_2$.
	Since $X \cap D_2 \subset X_2$, we see that $X_2$ separates $A_{n_i}$ from $A_{n_j}$ in $D_2$.
	Accordingly, choose points $x_1, \ldots, x_{4g+3}\in X_2$ so that 
	\begin{enumerate}
		\item each $x_i$ lies between $A'_{n_i}$ and $A'_{n_{i+1}}$ in $D_2$, and 
		\item the set $\bigcup_{j \neq i} A_{n_j}$ does not separate $x_i$ from $A_{n_i}$ in $D_2$.
	\end{enumerate}

  Since $X_2$ is connected and $X_2 \cap \bigcup_{j=1}^{4g+4} A_{n_j} = \emptyset$, there exists an arc $J \subset S \setminus \bigcup_{j=1}^{4g+4}A_{n_j}$ containing $\{x_i\}_{i=1}^{4g+3}$.
	By choice of $x_i$, there are disjoint arcs $\finsubseq{A''}{n}{i}{4g+3} \subset D_2$ which join $J$ to $A_{n_i}$ without intersecting any other $A'_{n_j}$.
	By collapsing the arc $J$ to a point, the graph $K_{3, 4g+3}$ is obtained in 
  \[ G = J \cup \bigcup_{i=1}^{4g+3}(\tilde{A}_{n_i} \cup A''_{n_i}), \] 
  i.e., $K_{3, 4g+3}$ is a minor of $G$.
	Lemma~\ref{lem:bipartite} concludes that the genus of $S$ is at least $g+1$, contradicting our assumption that the genus is $g$. 
\end{proof}

\subsection{A Partial Converse}

As can be expected, Theorem~\ref{thm:burgessSurface} has a partial converse: 
If $X$ is indecomposable, then there exists a sequence $(U_n)_{n=1}^\infty$ of complementary domains of $X$, not necessarily distinct, such that $\lim_{n \rightarrow \infty} \bd U_n = X$. 
This fact will be used in the next section, so we prove it here.  
The proof of this property follows closely the outline of \cite[Theorem 2.10]{CurMayTym07}, modified slightly to allow for a finite degree of \emph{multicoherence}. 

\begin{defn}[multicoherent]
	A connected topological space $X$ is \emph{multicoherent of degree $k$} if, for any pair of closed, connected subsets $A$ and $B$ such that $A \cup B = X$, the intersection $A \cap B$ consists of at most $k$ components.
\end{defn}

Closed surfaces and punctured closed surfaces are examples of finitely multicoherent spaces.  
Let $S$ be a finitely multicoherent space, and suppose that $A \subset S$ is a closed set which separates $p$ from $q$ for points $p, q \in S$.  
Then, by \cite[Theorem 1]{Stone:1949lr}, there is a closed subset $B \subset A$ which has at most $k$ components which also separates $p$ from $q$.  
We use this property in the following proof.
Here, if $A$ and $B$ are compact non-empty subsets of a metric space $(X,d)$, $H_d(A,B)$ represents the Hausdorff distance between $A$ and $B$ in the hyperspace of non-empty compact subsets of $X$.  See \cite{Nadler:1992yq} for details.

\begin{thm} \label{thm:convergeUp}
	Let $X$ be an indecomposable continuum in the closed surface $S$. Then there exists a sequence $(U_n)_{n=1}^\infty$ of (not necessarily distinct) complementary domains of $X$ such that $\lim_{n \rightarrow \infty} \bd U_n = X$. 
\end{thm}
\begin{proof}
	This is clear if $X$ is a point, so assume $X$ is a non-degenerate indecomposable continuum. 
	For the purposes of this proof, suppose that $S$ is equipped with a metric $d$ in which the set
	\[B_{\epsilon}(p) = \{x \in S \mid d(p, x) < \epsilon\}\]
	is simply connected when $\epsilon \in (0,1)$.
	Let $p$, $q$, and $r$ lie in different composants of $X$. 
	For each $n\in\nat$, define 
	\begin{align*}
    Q_n &= \text{the component of } X \setminus B_{1/n}(p) \text{ containing } q,\text{ and}\\
		R_n &= \text{the component of } X \setminus B_{1/n}(p) \text{ containing } r.
	\end{align*}
	
	Notice that $\lim_{n\rightarrow \infty}Q_n = \lim_{n \rightarrow \infty} R_n = X$, by density of composants. 
	Since $Q_n$ and $R_n$ are different components of $X \setminus B_{1/n}(p)$, they are separated in $S \setminus B_{1/n}(p)$ by $S \setminus (B_{1/n}(p) \cup X)$. 
	Thus, $Q_n$ and $R_n$ are closed and separated in the normal space $S \setminus B_{1/n}(p)$, so there is a subset $K_n$, closed in $S \setminus B_{1/n}(p)$, of $S \setminus (B_{1/n}(p) \cup X)$ which separates $Q_n$ and $R_n$. 
	Since $S \setminus B_{1/n}(p)$ is finitely multicoherent, say of degree $k$, then a subset 
		\[ L_n^1 \cup \ldots \cup L_n^k \] 
	is a closed separator (in $S \setminus B_{1/n}(p)$) of $Q_n$ and $R_n$, where the elements of the union are disjoint, closed, and connected (though perhaps some are empty). 
	We may assume that they are ordered so that 
	\[H_d(L_n^1, X) \le \ldots \le H_d(L_n^{k}, X),\]
	where $H_d(\emptyset, X)$ can be regarded as $\infty$.
	Moreover, since each $L_n^i$ is disjoint from $X$ and connected, there exists a complementary domain $U_n^i$ such that
	\[L_n^i \subset U_n^i.\]
	(The set $U_n^i$ can be any complementary domain of $X$ if $L_n^i$ is empty.) 
	
	The sequence $(U_n^1)_{n=1}^\infty$ formed in this way is the required sequence of complementary domains.  
	We will show that any convergent subsequence of $(\bd U_n^1)_{n=1}^\infty$ converges to $X$, implying that the sequence itself converges to $X$.
	
	First we demonstrate that $X \subset \liminf_{n \rightarrow \infty} \bigcup_{i=1}^k \bd U^i_n$. 
	Choose $x \in X \setminus \{p\}$ and $0 < \epsilon < d(x,p)$.  
	Let $N \in \nat$ such that, for all $n \ge N$ 
	\begin{enumerate}
		\item $Q_n \cap B_\epsilon(x) \neq \emptyset$, 
		\item $R_n \cap B_\epsilon(x) \neq \emptyset$, and 
		\item $B_{1/n}(p) \cap B_\epsilon(x) = \emptyset$. 
	\end{enumerate}
	
	For $n \ge N$, choose $q_n \in Q_n \cap B_\epsilon(x)$ and $r_n \in R_n \cap B_\epsilon(x)$. 
	Let $A_n \subset B_\epsilon(x) \subset S \setminus B_{1/n}(p)$ be an arc joining $q_n$ to $r_n$.
	Then
	\[A_{n} \cap \bigcup_{i=1}^k L_n^i \neq \emptyset,\]
	since $\bigcup_{i=1}^k L_n^i$ separates $q_n$ from $r_n$ in $S \setminus B_{1/n}(p)$. 
	Since $L_n^i \subset U_n^i$ for each $i$ and $q_n$, $r_n$ are not in any $U_n^i$ (they lie in $X$), 
	\[A_{n} \cap \bigcup_{i=1}^k \bd U_n^i \neq \emptyset,\] implying that 
	\[\bigcup_{i=1}^k \bd U_n^i \cap B_{\epsilon}(x) \neq \emptyset.\]
	This is true for all $n \ge N$, so $x \in \liminf_{n \rightarrow \infty} \bigcup_{i=1}^k \bd{U^i_n}$, and we have that $X = \lim_{n \rightarrow \infty} \bigcup_{i=1}^k\bd{U^i_n}$.
	
	Now, let us consider the individual limits $\lim_{n \rightarrow \infty} \bd U_n^i$ for some fixed $i \le k$. 
	By passing to a subsequence, we may assume that the limit $X_i = \lim_{n \rightarrow \infty} \bd U_n^i$ exists for each $i \le k$.  
	Then, since $\lim_{n \rightarrow \infty} \bigcup_{i=1}^k \bd U_n^i = X$, we see that $\bigcup_{i=1}^k X_i = X$. 
	However, this is a finite union of continua, and $X$ is indecomposable, so at least one $X_i$ is not a proper subcontinuum of $X$.  
	By continuity, we see that $H_d(X_1, X) \le H_d(X_i, X)$ for all $i \le k$, so $X_1 = \lim_{n\rightarrow \infty} \bd U_n^1 = X$.  
\end{proof}

\section{Characterization of Shielded Indecomposable Continua}\label{sec:charThm} 
A subcontinuum $X$ of a closed surface $S$ is called $\emph{shielded}$ if, for every complementary domain $U$ of $X$, $\bd U \neq X$.
We extend Theorem~\ref{thm:planchar}, which is a characterization of planar indecomposable continua, to shielded subcontinua of closed surfaces.

The assumption that a continuum is shielded already imparts some complexity.  
For instance, any sequence of complementary domains whose boundaries converge to the continuum must consist of infinitely many distinct elements.
We use this property to bridge the gap left by the relatively weak separation properties of compact surfaces.

The proofs in this section will depend upon the existence of particularly well-behaved homeomorphisms of simply connected domains in a surface to the unit disk $\Disk$ in the plane.
Specifically, we wish for a null sequence of crosscuts in a domain to correspond to a null sequence of crosscuts in $\Disk$. 
We will prove that conformal isomorphisms have this property, so for the remainder we will assume that the closed surface $S$ is endowed with a conformal structure and that $d$ is the corresponding metric.

\begin{lem}
  If $X$ is a non-degenerate continuum in $S$, then all of its simply connected complementary domains are conformally isomorphic to the unit disk.
\end{lem}
\begin{rem}
  Note that, in contrast to the planar case, some complementary domains may not be simply connected. 
\end{rem}
\begin{proof}
  Let $P$ be the universal covering space of $S$, with corresponding conformal covering map $\pi:P \rightarrow S$.
  Recall that, by the Uniformization Theorem \cite[Theorem 1.1]{Mil06}, $P$ is conformally isomorphic to a simply connected subset of the Riemann sphere.
  If $\hat U$ is a component of $\pi^{-1}(U)$, then $\hat U$ is simply connected and $\pi|_{\hat U}$ is a covering map.  In fact, $\pi|_{\hat U}$ is a conformal isomorphism since its trivial fundamental group is isomorphic to the group of deck transformations for $\pi|_{\hat U}$.
  However, $\hat U$ misses $\pi^{-1}(X)$, so $\hat U$ is conformally isomorphic to the unit disk by the classical Riemann mapping theorem.
\end{proof}

The proof of the following lemma is identical to the proof of \cite[Lemma 3.4]{CurMayTym07}, and is included here for completeness. 
\begin{lem}
	Let $U$ be a simply connected open subset of $S$ with non-degenerate boundary. 
	Let $\phi:U \rightarrow \Disk$ be a conformal isomorphism. 
	Then the image of a null sequence $(K_n)_{n=1}^\infty$ of crosscuts of $U$ is a null sequence of crosscuts in $\Disk$. 
\end{lem}
\begin{proof}
	Let $(K_n)_{n=1}^\infty$ be a null sequence of crosscuts of $U$ with image sequence $(A_n)_{n=1}^\infty = (\phi(K_n))_{n=1}^\infty$.
	Without loss of generality, assume that $(K_n)_{n=1}^\infty$ converges to a point $x \in \bd U$.
	By passing to a subsequence, we may assume that the image sequence converges to a continuum $L \subset \overline \Disk$. 
	Since $(K_n)_{n=1}^\infty$ does not accumulate on a subset of $U$, we see that $L \subset \bd \Disk$. 
	
	Let $t \in L$. 
	There exists a chain of crosscuts $(A_n')_{n=1}^\infty$ of $\Disk$ converging to $t$ which maps to a null sequence $(K_n')_{n=1}^\infty$ of crosscuts of $U$ by $\phi^{-1}$.
	(See \cite[Lemma 17.9]{Mil06}; the proof does not rely on the planarity of $U$.)
	We may assume that $(K_n')_{n=1}^\infty$ converges to a point of $\bd U$ by passing to a subsequence. 
	For each $m \in \nat$, that $(A_n)_{n=1}^\infty$ accumulates on $t$ implies that all but finitely many $A_n$ intersect the crosscut neighborhood of $A_m'$ corresponding to $t$. 
	Also, since $(K_n)_{n=1}^\infty$ forms a null sequence in $U$ and converges to $x$, we see that all but finitely many $K_n$ (thus $A_n$) lie entirely within the crosscut neighborhood of $K_m'$ (thus $A_m'$) corresponding to $t$. 
	However, the crosscut neighborhoods of $A_m'$ form a null sequence as $m \rightarrow \infty$, so $(A_n)_{n=1}^\infty$ form a null sequence. 
\end{proof}
\begin{lem}\label{lem:unlinked} 
	Suppose $U \subset S$ is open, connected, and simply connected with non-degenerate boundary, and let $\phi:U \rightarrow \Disk$ be a conformal isomorphism. 
	Let $B_1$ and $B_2$ be disjoint closed disks meeting $\bd U$, and let $E_i \subset \bd \Disk$ denote the set of endpoints of the crosscuts of $\Disk$ which form $\phi((\bd B_i) \cap U)$. 
  If, for different $i$ and $j$, $E_i$ lies in a component of $\ucirc \setminus E_j$, then there is a generalized crosscut $K$ of $U$ which separates $B_1 \cap U$ from $B_2 \cap U$ in $U$. 
	Moreover, if $\bd U$ is locally connected, then $K$ is a crosscut of $U$. 
\end{lem}
\begin{proof}
	Identical to proof of Lemma~3.5 in \cite{CurMayTym07}. 
\end{proof}

\begin{thm}\label{thm:surfaceChar}
	\label{thm:dblpass} A shielded subcontinuum $X$ of a closed surface $S$ is indecomposable if and only if it has a sequence of complementary domains which satisfies the double-pass condition. 
\end{thm}
\begin{proof}
	Suppose first that $X$ is indecomposable. 
	By Theorem~\ref{thm:convergeUp}, there is a sequence $\seq U n$ of complementary domains of $X$ such that $\lim_{n \rightarrow \infty} \bd U_n = X$.
	For each $n \in \nat$, let $K_n$ be a generalized crosscut of $U_n$. 
	Let $A_n$ and $B_n$ be shadows of $K_n$ so that $A_n \cup B_n = \bd U_n$, with $H_d(A_n, X) \le H_d(B_n, X)$.

	Then $A_n$ and $B_n$ are subcontinua of $X$. 
	There is a subsequence $\subseq U n i$ so that $\subseq A n i$ and $\subseq B n i$ converge to subcontinua $A$ and $B$ of $X$. 
	We see that $A \cup B = X$, since $A_{n_i} \cup B_{n_i} = \bd U_{n_i}$. 
	Since $X$ is indecomposable, either $A$ or $B$ is not a proper subcontinuum of $X$.
	We have $H_d(X, A_{n_i}) \le H_d(X, B_{n_i})$, so $A = X$. 
	This is true for all choices of subsequences $\subseq U n i$ where $\subseq A n i$ and $\subseq B n i$ both converge, so $\seq U n$ satisfies the double-pass condition.
	
	Now, assume that a continuum $X \subset S$ has a sequence $\seq U n$ of complementary domains satisfying the double-pass condition. 
	By way of contradiction, suppose that $(X_1, X_2)$ is a decomposition of $X$. 
	There are disjoint closed disks $D_1,D_2 \subset S$ so that, for $i\in \{1,2\}$, 
	\begin{enumerate}
		\item the interior of $D_i$ intersects $X_i$, and 
		\item if $i \neq j$, $D_i$ is disjoint from $X_{j} \cup D_{j}$. 
	\end{enumerate}
	Since $X$ satisfies the double-pass condition, there exists $N \ge 1$ such that, for all $n \ge N$, no generalized crosscut of $U_n$ separates $D_1 \cap U$ from $D_2 \cap U$ in $U_n$. Without loss of generality, $N = 1$.
	
	Now we make use of the assumption that $X$ is shielded. 
	Since $\bd U_n \neq X$ for each $n$ and $\lim_{n \rightarrow \infty} \bd U_n = X$, no element of $\seq U n$ appears infinitely often in the sequence. 
	Therefore, we can assume by passing to a subsequence that $\seq U n$ consists of different simply connected complementary domains of $X$.  
	
	Let $\phi_n:U_n \rightarrow \Disk$ be a conformal isomorphism.  According to Lemma~\ref{lem:unlinked}, the sets $E_{n,1}$ and $E_{n,2}$, comprised of the endpoints of the crosscuts constituting $\phi_n((\bd D_1) \cap U_n)$ and $\phi_n((\bd D_2) \cap U_n)$ separate each other in $\ucirc$.
	It is evident that $D_1 \cap U_n$ and $D_2 \cap U_n$ may or may not separate the other in $U_n$.
	By passing to a subsequence, we can assume that, for every $n \in \nat$, either 
	\begin{enumerate}
	\item $D_1 \cap U_n$ separates $D_2 \cap U_n$ in $U_n$, or 
	\item neither separates the other in $U_n$.
	\end{enumerate}
	We will find a crosscut $F_n \subset \Disk \setminus \phi(D_2)$ of $\Disk$ which joins points of $E_{n,1}$ which separate $E_{n,2}$ in $\ucirc$.
	In the case that $D_1 \cap U_n$ separates $D_2 \cap U_n$, a component of $(\bd D_1) \cap U_n$ also does, so we can define $F_n$ as a component of $\phi((\bd D_1) \cap U_n)$ which separates $\phi_n(D_2 \cap U_n)$ in $\Disk$. 
	In the second case, there are components $K_1$ and $K_2$ of $\phi_n((\bd D_1) \cap U_n)$ whose endpoints separate $E_{2,n}$ in $\ucirc$.
	Let $F_n$ be the union of an arc joining $k_1 \in \Int(K_1)$ to $k_2 \in \Int(K_2)$ with one component each of $K_1 \setminus \{k_1\}$ and $K_2 \setminus \{k_2\}$.
	Notice that $\phi_n^{-1}(F_n)$ is in fact a crosscut of $U_n$ in either case, since points of $F_n$ close to $\bd U_n$ are in $\bd D_1$.

	Let $A_n$ be an arc in $U_n$ joining components of $D_2 \cap U_n$ which are separated by $\phi_n^{-1}(F_n)$, intersecting $\phi_n^{-1}(F_n)$ transversely exactly once. 
	Let $C_n \subset D_2 \cup A_n$ be a simple closed curve containing $A_n$, formed by joining the endpoints of $A_n$ to a designated point $a \in D_2$ with line segments otherwise disjoint from $C_1 \cup \ldots C_{n-1}$.
	
	Let $Y = C_1 \cup C_2 \cup \ldots \cup C_{2g+1}$, where $g$ is the genus of $S$.
	This is the one-point union of $2g+1$ simple closed curves, so $S \setminus Y$ is not connected (see for instance \cite[5.14]{Aleksandrov:1956fk}). 
	Since the arcs $F_1, \ldots, F_{2g+1}$ intersect $Y$ exactly once transversely and have their endpoints in $X_2$, we conclude that each component of $S \setminus Y$ contains points of $X_1$.
	However, since $Y$ is the union of arcs in $S \setminus X$ and $B_2$, we see that $Y$ is disjoint from $X_1$. 
	This contradicts the assumption that $X_1$ is connected.
\end{proof}

\begin{rem}
	Suppose that the continuum in question has locally connected boundary components.  
	Then Lemma~\ref{lem:unlinked} provides a crosscut (rather than a generalized crosscut) to show that $X$ fails the double-pass condition.
	Hence, we can observe that a continuum with locally connected boundary components is indecomposable if and only if it satisfies the double-pass condition with crosscuts.
\end{rem}

\section{Conclusion}\label{sec:conclusion} 
Though the results presented here are advances, the corresponding theorems for planar continua are much stronger.  
Here we discuss stronger generalizations which may hold.

\subsection{Burgess's Theorem}
Theorem~\ref{thm:burgessSurface} is a direct translation of the corollary to \cite[Theorem 10]{Bur51}. 
In turn, that theorem follows from a fundamental theorem of C.~E. Burgess, reproduced below.
\begin{thmn}{\cite[Theorem 9]{Bur51}}
	Let $H \subset \sphere$ be closed, $M \subset \sphere$ a continuum, and $M_1$, $M_2$, and $M_3$ subcontinua of $M$. 
	Suppose that $K_1$, $K_2$, and $K_3$ are closed disks, disjoint from each other and $H$, and $K_i$ intersects $M_j$ if and only if $i = j$. 
	Then there do not exist three complementary domains of $M \cup H$ all of which intersect $K_1$, $K_2$, and $K_3$. 
\end{thmn}

\begin{ques}\label{ques:BurgessDisks}
  Is there an extension of this theorem to closed surfaces?
\end{ques}

Using Burgess's theorem, one can prove that a planar continuum which is the common boundary of three of its complementary domains \cite{Kur68} or is the impression of one of its prime ends \cite{Rut35} is either indecomposable or 2-indecomposable. 
Also, a planar continuum which is the limit of a disjoint sequence of shadows is indecomposable \cite{CurMayTym07}. 
Theorem~\ref{thm:burgessSurface} was proven here without proving the analog of Burgess's theorem, but interesting theorems would surely follow if Question~\ref{ques:BurgessDisks} had a positive answer.

\subsection{Kuratowski's Theorem}
In particular, the requirement of Theorem~\ref{thm:burgessSurface} that $X$ have infinitely many distinct domains whose boundaries limit to the continuum is probably stronger than necessary for the conclusion. This motivates the following question. 
\begin{ques}
	Is there a finite version of Theorem~\ref{thm:burgessSurface} like Kuratowski's theorem?  Specifically, if a continuum $X$ in a closed surface of genus $g$ is the common boundary of three complementary domains, what is the maximal $n(g)$ so that $X$ is the union of $n(g)$ indecomposable continua?
\end{ques}

A partial answer is illustrated by an example.  Let $C_1, C_2, C_3 \subset \rsphere \setminus \{\infty\}$ be homeomorphic copies of the pseudocircle such that, for distinct $i, j, k$,
\begin{enumerate}
	\item $C_i$ is contained in the closure of the component of $\sphere \setminus C_j$ containing $\infty$, 
	\item $C_i \cap C_j$ is a single point not in $C_k$, and
	\item \label{item:composant} $C_i \cap C_k$ is in a different composant of $C_k$ than $C_j \cap C_k$.
\end{enumerate}
This is a continuum with $5$ complementary domains.  
Two complementary domains, $U$ and $V$, meet each element of $\{C_1, C_2, C_3\}$.  
There are three complementary domains, $W_1$, $W_2$, and $W_3$, where $W_i$ is the bounded complementary domain of $C_i$.  

No proper subcontinuum of $C_1 \cup C_2 \cup C_3$ can separate $U$ from $V$ because of condition~\ref{item:composant}, so $\bd U = \bd V = C_1 \cup C_2 \cup C_3$.
Remove disks $D_1, D_1' \subset W_1$, $D_2 \subset W_2$, and $D_3 \subset W_3$, and paste cylinders joining $\bd D_1$ to $\bd D_2$ and $\bd D_1'$ to $\bd D_3$ in a way resulting in an orientable surface $S$ of genus 2.
Then $C_1 \cup C_2 \cup C_3$ is the common boundary of three domains, though it is not the union of just two indecomposable continua.   
This indicates that an extension of Kuratowski's theorem along these lines must take the genus into account.

\subsection{Characterization Theorem}
Theorem~\ref{thm:dblpass} is significantly weaker than Theorem~\ref{thm:planchar}, since Theorem~\ref{thm:planchar} holds for continua which are boundaries of connected open sets in $\sphere$.
The author suspects the statement of Theorem~\ref{thm:planchar} holds true in surfaces, with some qualification. 
For instance, an essentially embedded simple closed curve in a torus satisfies the double-pass condition as it is stated here, but artificially so. 

One can define a \emph{generalized crosscut with curves} of a domain $U$ as the disjoint union of a generalized crosscut and a finite number of simple closed curves.  By defining shadows in analogy to shadows of generalized crosscuts, one obtains an equivalent notion in simply connected domains -- adding disjoint simple closed curves to a generalized crosscut does not change the shadows.
However, generalized crosscuts with curves allows stronger separation in domains which are not simply connected.  
 
\begin{ques}
	Let $X$ be a continuum in a closed surface $S$. If for every sequence $(C_i)_{i=1}^\infty$ of generalized crosscuts with curves there exists a choice of shadows $(S_i)_{i=1}^\infty$, where $S_i$ a shadow of $C_i$, which converge to $X$, is $X$ indecomposable? 
\end{ques}

\subsection{Higher Dimensions}
One may ask how useful recognition from the complement might be in more general spaces.  
Specifically, in $\mathbb R^3$, this approach does not look useful. 
Formulations of prime end theory in $\mathbb R^3$ have had limited applicability.  
Further, M.~Luba{\'n}ski constructed in \cite{Lub53} a family of absolute neighborhood retracts which can be the common boundary of any finite number of domains in $\mathbb R^3$.
Kuratowski \cite[p. 560]{Kur68} noted that this can be extended to an infinite number of complementary domains.

\bibliographystyle{plain}

\begin{thebibliography}{10}

\smallskip

\bibitem{Aleksandrov:1956fk}
P.~S. Aleksandrov.
\newblock {\em Combinatorial topology. {V}ol. 1}.
\newblock Graylock Press, Rochester, N. Y., 1956.

\bibitem{Bur51}
C.~E. Burgess.
\newblock Continua and their complementary domains in the plane.
\newblock {\em Duke Math. J.}, 18:901--917, 1951.

\bibitem{CurMayTym07}
Clinton~P. Curry, John~C. Mayer, and E.~D. Tymchatyn.
\newblock Characterizing indecomposable plane continua from their complements.
\newblock Proc. Amer. Math. Soc., 2008, to appear.

\bibitem{Frank1969}
Frank Harary.
\newblock {\em Graph theory}.
\newblock Addison-Wesley Publishing Co., Reading, Mass.-Menlo Park,
  Calif.-London, 1969.

\bibitem{HocYou61}
J.~Hocking and G.~Young.
\newblock {\em Topology}.
\newblock Addison-Wesley, Reading, MS, 1961.

\bibitem{Kra74}
J.~Krasinkiewicz.
\newblock On internal composants of indecomposable plane continua.
\newblock {\em Fund. Math.}, 84:255--263, 1974.

\bibitem{Kur68}
K.~Kuratowski.
\newblock {\em Topology}, volume~II.
\newblock Academic Press, New York, 1968.

\bibitem{Lub53}
M.~Luba{\'n}ski.
\newblock An example of an absolute neighbourhood retract, which is the common
  boundary of three regions in the {$3$}-dimensional {E}uclidean space.
\newblock {\em Fund. Math.}, 40:29--38, 1953.

\bibitem{Mil06}
J.~Milnor.
\newblock {\em Dynamics in One Complex Variable}.
\newblock The Annals of Mathematics Studies. Princeton University Press,
  Princeton, NJ, 3rd edition, 2006.

\bibitem{Nadler:1992yq}
S.~B. Nadler, Jr.
\newblock {\em Continuum theory}, volume 158 of {\em Monographs and Textbooks
  in Pure and Applied Mathematics}.
\newblock Marcel Dekker Inc., New York, 1992.
\newblock An introduction.

\bibitem{Rut35}
N.~E. Rutt.
\newblock Prime ends and indecomposability.
\newblock {\em Bull. Amer. Math. Soc.}, 41:265--273, 1935.

\bibitem{Stone:1949lr}
A.~H. Stone.
\newblock Incidence relations in multicoherent spaces. {I}.
\newblock {\em Trans. Amer. Math. Soc.}, 66:389--406, 1949.

\end{thebibliography}

\end{document}